\documentclass[preprint,review, 10pt,3p,times]{elsarticle}

\usepackage{amsthm,amsmath,amsfonts,graphicx}

\usepackage[english]{babel}
\usepackage{amsmath,amssymb}
\usepackage{amsfonts}
\usepackage{latexsym}
\usepackage{epsfig}
\usepackage{graphicx}
\usepackage{url}
\usepackage[T1]{fontenc}
\usepackage{color}

\newcommand{\bc}{\begin{center}}
\newcommand{\ec}{\end{center}}
\newtheorem{theorem}{Theorem}
\newtheorem{proposition}{Proposition}
\newtheorem{corollary}{Corollary}
\newtheorem{lemma}{Lemma}

\begin{document}

\begin{frontmatter}

\title{Nordhaus-Gaddum bounds for locating domination}

\author[upc]{C.~Hernando}
\ead{carmen.hernando@upc.edu}

\author[upc]{M.~Mora}
\ead{merce.mora@upc.edu}

\author[upc]{I. M.~Pelayo\corref{cor1}}
\ead{ignacio.m.pelayo@upc.edu}

\cortext[cor1]{Corresponding author}

\address[upc]{Universitat Politècnica de Catalunya, Barcelona, Spain}

\author{}

\address{}

\begin{abstract}
\small A dominating set $S$  of  graph $G$ is called \emph{metric-locating-dominating} if it is also locating, that is,  if every vertex $v$  is uniquely determined by its vector of distances to the vertices in $S$. If moreover, every vertex $v$ not in  $S$ is also uniquely determined by the set of neighbors of $v$ belonging to $S$,  then it is said to be \emph{locating-dominating}. Locating, metric-locating-dominating and locating-dominating sets of minimum cardinality are called $\beta$-codes, $\eta$-codes and $\lambda$-codes, respectively. A Nordhaus-Gaddum bound is a tight lower or upper bound on the sum or product of a parameter of a graph $G$  and its complement $\overline G$. In this paper, we present some Nordhaus-Gaddum bounds for the location number $\beta$, the metric-location-domination number $\eta$ and the location-domination number $\lambda$. Moreover, in each case, the graph family attaining the corresponding bound is fully characterized.
\end{abstract}


\end{frontmatter}

\section{Introduction}


Given a graph $G=(V,E)$,
the \emph{(open) neighborhood} of a vertex $v\in V$ is $N_G(v)=N(v)=\{u\in V : uv\in E\}$. The distance between vertices $v,w\in V$ is denoted by $d_G(v,w)$, or $d(v,w)$ if the graph G is
clear from the context. The diameter $diam(G)$ is the maximum distance between any two vertices of $G$. Let $S=\{x_1,\ldots,x_k\}$ be a set of vertices  and let $v\in V\setminus S$. The ordered $k$-tuple $c_{\stackrel{}{S}}(v)=(d(v,x_1),\ldots,d(v,x_k))$ is called the vector of \emph{metric coordinates} of $v$ with respect to $S$. For further notation see \cite{chlezh11}.

A set $D\subseteq V$ is a \emph{dominating set} is for every vertex $v\in V\setminus D$, $N(v)\cap D\neq\emptyset$.
The \emph{domination number} $\gamma(G)$ is the minimum cardinality of a dominating set of $G$. A dominating set of cardinality $\gamma(G)$ is called a \emph{$\gamma$-code} \cite{hahesl98}. 
A set $D=\{x_1,\ldots,x_k\}\subseteq V$  is  a 
\emph{locating set} if for every pair of distinct vertices $u,v\in V$, 
$c_{\stackrel{}{D}}(u)\neq c_{\stackrel{}{D}}(v)$.
The \emph{location number}  (also called the \emph{metric dimension}) $\beta(G)$
is the minimum cardinality of a locating set of $G$~\cite{hame,slater75}. A locating set of cardinality $\beta(G)$ is called a \emph{$\beta$-code}. A \emph{metric-locating-dominating set}, a MLD-set for short, is any set of vertices that is both a dominating set and a locating set. The \emph{metric-location-domination number} $\eta(G)$ is the minimum cardinality of a metric-locating-dominating set of $G$. A metric-locating-dominating set of cardinality $\eta(G)$ is called a \emph{$\eta$-code}~\cite{heoe}. 
A set $D\subseteq V$ is  a \emph{locating-dominating set}, an LD-set for short,  if for every two vertices  $u,v\in V(G)\setminus D$, $\displaystyle \emptyset\neq N(u)\cap D\neq N(v)\cap D\neq\emptyset.$
The \emph{location-domination number}  $\lambda(G)$  
is the  minimum cardinality of a locating-dominating set.  A locating-dominating set of cardinality $\lambda(G)$ is called a \emph{$\lambda$-code}~\cite{slater88}. A complete and regularly updated list of papers
on locating dominating codes is to be found in \cite{lobstein}.

Clearly, every locating-dominating set is locating and also dominating. Moreover, both location and domination are hereditary properties. Particularly, if for two sets   $S_1,S_2\subset V$,  $S_1$ is locating and $S_2$ is dominating, then  $S_1\cup S_2$ is both locating and dominating.  Hence, for every graph $G$, 
$\max\{\gamma(G),\beta(G)\} \leq \eta(G) \leq \min\{\gamma(G)+\beta(G),\lambda(G)\}$ \cite{cahemopepu12}.

A Nordhaus-Gaddum bound is a tight lower or upper bound on the sum or product of a parameter of a graph $G$  and its complement $\overline G$ \cite{aoha12,hejoso11,jowo}. For example, in \cite{cohe77}  it was shown that for any graph $G$ of order $n$, $\gamma(G)+\gamma(\overline G)\le n+1$, the equality being true only if $\{G,\, \overline G\}=\{K_n,\, \overline K_n\}$. In this paper,  we present some Nordhaus-Gaddum bounds on the sum of  the location number $\beta$, the metric-location-domination number $\eta$ and the location-domination number $\lambda$. In all cases, the classes of graphs attaining both bounds are characterized.

\section{Nordhaus-Gaddum bounds}\label{valoresextremos}

Unless otherwise stated, along this section $G=(V;E)$ is a, not necessarily connected, nontrivial graph of order $n$.
A graph $G$ is called \emph{doubly-connected} if both $G$ and its complement $\overline G$ are connected. As usual, $K_n$, $C_n$ and $P_n$ denote respectively the complete graph, the cycle and the path on $n$ vertices.

%

\subsection{Location number}


\begin{theorem}\label{beta1} For every  nontrivial graph $G$, $2\le \beta(G)+\beta(\overline G)\le 2n-1$. Moreover,
\begin{itemize}
\item $\beta(G)+\beta(\overline G)=2$ if and only if $G=P_4$.
\item $\beta(G)+\beta(\overline G)=2n-1$ if and only if $\{G,\, \overline G\}=\{K_n,\, \overline K_n\}$.
\end{itemize}
\end{theorem}
\begin{proof} Every graph satisfies $1\le\beta(G)$, which means that $2\le \beta(G)+\beta(\overline G)$. Moreover, the equality $\beta(G)+\beta(\overline G)=2$ is only true for $G=P_4$, since paths $P_n$ are the only graphs with location number 1 \cite{cherjooe00}, and  $P_4=\overline  P_4$ is the only nontrivial path whose complement is also a path.
The upper bound immediately follows from these facts: (1) the  graph $\overline K_n$ is the only graph with location number $n$ and (2) $\beta(K_n)=n-1$. Finally, claims (1) and (2) also allows us to derive that equality  $\beta(G)+\beta(\overline G)=2n-1$ only holds when $\{G,\, \overline G\}=\{K_n,\, \overline K_n\}$.
\end{proof}

\begin{figure}[ht]
  \begin{center}
        \includegraphics[width=.7\textwidth]{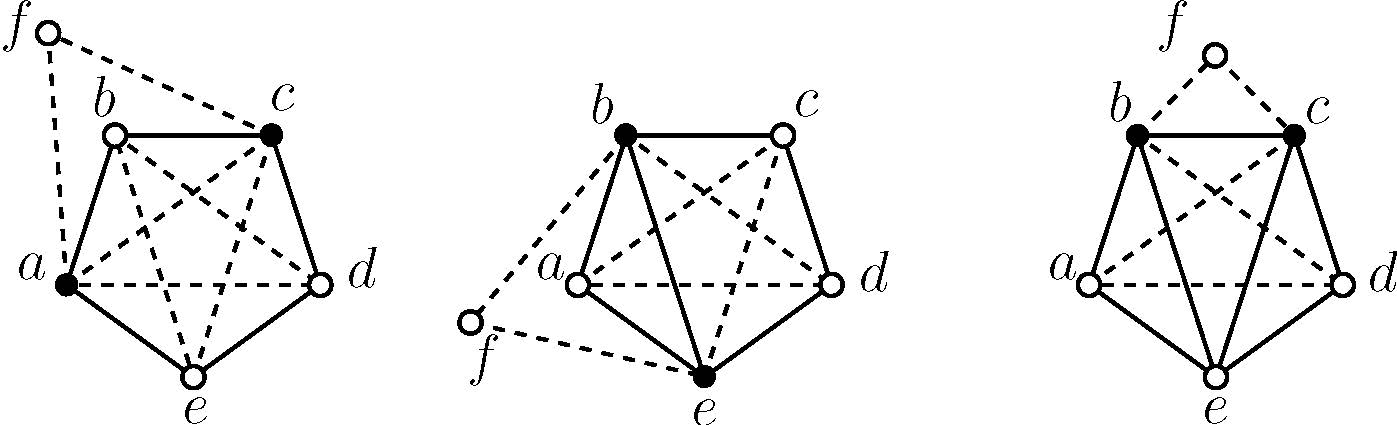}
  \end{center}
  \caption{Solid lines are edges in $G$ and dashed lines are edges in $\overline G$.}
  \label{casos}
\end{figure}

\begin{lemma}\label{n-3D2} Every  doubly-connected graph $G$  of order $n\ge6$ such that $diam(G)=diam(\overline G)=2$  contains a locating set of cardinality $n-4$.
\end{lemma}
\begin{proof} 

Let $\rho$ be an induced path of order 4 in $G$, whose existence is guaranteed since, as was proved in \cite{colebu81}, the complement of every nontrivial $P_4$-free graph is not connected. 
Assume that $V(\rho)=\{a,b,c,d\}$ and $E(\rho)=\{ab,bc,cd\}$. Since, $diam(G)=2$, there exists a vertex $e\in V(G)$ such that $d_G(a,e)=d_G(e,d)=1$. 
We distinguish three cases.

{\bf Case 1:} $eb, ec\not \in E(G)$ (see Figure\ref{casos}, left). In this case, the set $\{a,b,c,d,e\}$ determines an induced cycle $\Gamma$ in $G$ and also an induced cycle $\overline \Gamma$ in $\overline G$. Let $f$ a vertex not belonging to $\{a,b,c,d,e\}$. Either in $G$ or in $\overline G$, $f$ has at most two neighbors in $\{a,b,c,d,e\}$.
Without loss of generality we may suppose that $N_G(f)\cap \{a,b,c,d,e\}\le2$ (otherwise we interchange labels $G$ and $\overline G$), which means that  there exist in $\{a,b,c,d,e\}$ a pair of non-consecutive vertices non-adjacent to $f$. Again w.l.o.g. we assume that $N_G(f)\cap \{a,c\}=\emptyset$. Certainly, the set $V(G)\setminus\{b,d,e,f\}$ is a locating set of $G$ since $c_{\stackrel{}{\{a,c\}}}(b)=(1,1)$, $c_{\stackrel{}{\{a,c\}}}(d)=(2,1)$, $c_{\stackrel{}{\{a,c\}}}(e)=(1,2)$ and 
$c_{\stackrel{}{\{a,c\}}}(f)=(2,2)$.

{\bf Case 2:} $e$ is adjacent to exactly one vertex of $\{b,c\}$. Let us assume that $eb\in E(G)$ and $ec\not\in E(G)$ (see Figure\ref{casos}, center). In this case, $d_G(e,b)=1$, which means that $d_{\overline G}(e,b)=2$ since $diam(\overline G)=2$.
Therefore, there exists a vertex $f\not\in\{a,b,c,d,e\}$ such that $d_{\overline G}(e,f)=d_{\overline G}(f,b)=1$. 
This means that  $d_{G}(e,f)=d_{G}(f,b)=2$ as $diam(\overline G)=2$. Hence, the set $V(G)\setminus\{a,c,d,f\}$ is a locating set of $G$ since $c_{\stackrel{}{\{b,e\}}}(a)=(1,1)$, $c_{\stackrel{}{\{b,e\}}}(c)=(1,2)$, $c_{\stackrel{}{\{b,e\}}}(d)=(2,1)$ and 
$c_{\stackrel{}{\{b,e\}}}(f)=(2,2)$.

{\bf Case 3:} $eb, ec \in E(G)$ (see Figure\ref{casos}, right). Since $d_G(b,c)=1$, we have $d_{\overline G}(b,c)=2$. Therefore, there exists a vertex $f\not\in\{a,b,c,d,e\}$ such that $d_{\overline G}(b,f)=d_{\overline G}(f,c)=1$.
This means that  $d_{G}(b,f)=d_{G}(f,c)=2$. Hence, the set $V(G)\setminus\{a,d,e,f\}$ is a locating set of $G$ since $c_{\stackrel{}{\{b,c\}}}(a)=(1,2)$, $c_{\stackrel{}{\{b,c\}}}(d)=(2,1)$, $c_{\stackrel{}{\{b,c\}}}(e)=(1,1)$ and 
$c_{\stackrel{}{\{b,c\}}}(f)=(2,2)$.
\end{proof}

Take a connected graph  $G$ of order $n$, and assume that $V(G)=\{1,\ldots,n\}$. Let $G[H^{(i)}]$ denote the graph obtained from $G$ by replacing vertex $i$ by a given graph $H$ and joining every vertex of $H$ to every  neighbor of vertex $i$ in $G$. Similarly, $G[H_1^{(i)},H_2^{(j)}]$ denotes the graph obtained from $G$ by replacing  vertex $i$ by a graph $H_1$ and vertex $j$ by a graph $H_2$ and joining every vertex of $H_1$ (resp. vertex of $H_2$) to every  neighbor of vertex $i$ (resp. $j$) in $G$ and, just if $ij\in E(G)$,  also every 
vertex of $H_1$ to every vertex  of $H_2$. Finally, $B$ denotes the bull graph shown in Figure \ref{hbdf}.

\begin{theorem}\label{beta2} For any doubly-connected graph $G$ with $n\geq4$,  $2\le \beta(G)+\beta(\overline G)\le 2n-6$. Moreover,
\begin{itemize}
\item $\beta(G)+\beta(\overline G)=2$ if and only if $G=P_4$.
\item $\beta(G)+\beta(\overline G)=2n-6$ if and only if 
$G\in\Omega_1\cup\Omega_2\cup\Omega_3$, where
\begin{itemize}
\item $\Omega_1=\{P_4, C_5, B\}$
\item $\Omega_2=\{P_4[K_{n-3}^{(1)}], P_4[\overline K_{n-3}^{(1)}], P_4[K_{n-3}^{(2)}], P_4[\overline K_{n-3}^{(2)}]\}$
\item $\Omega_3=\{P_4[K_{r}^{(1)},K_{n-r-2}^{(2)}]:1\le r \le n-3\} \cup \{P_4[\overline K_{r}^{(1)},\overline K_{n-r-2}^{(3)}]:1\le r \le n-3\}$
\end{itemize}
\end{itemize}
\end{theorem}
\begin{proof}
In \cite{cherjooe00}, it was proved that a connected graph $G$ satisfies $n-2\le \beta(G) \le n-1$ if and only if, for some  $1\le h\le n-1$, $G\in\{K_n, K_{h,n-h},K_h+\overline K_{n-h}, K_h+(K_1\cup K_{n-h-1}\}$. 
It is a routine exercise to check that the complement  of  any of these graphs is not connected. Hence,  every doubly-connected graph $G$ of order $n\geq4$ satisfies $1\le\beta(G)\le n-3$, i.e., $2\le \beta(G)+\beta(\overline G)\le 2n-6$. Moreover, according to Theorem \ref{beta1}, the lower bound 2 is attained only for $G=P_4$, since $\overline P_4=P_4$. 

Let  $G$ be a doubly-connected graph of order $n\ge4$  verifying $\beta(G)+\beta(\overline G)=2n-6$, i.e., such that $\beta(G)=\beta(\overline G)=n-3$. In \cite{cherjooe00}, it was proved that the order of a graph $G$ of diameter $D$ and location number $\beta$ is at least $\beta+D$. This means, that if $\beta(G)=n-3$, then $2\le D\le 3$, since $\beta(K_n)=n-1$. In \cite{ours2}, the set of graphs with $n$ vertices, diameter $D$ and location number $n-D$ were characterized for all feasible values of $n$ and $D$. In particular, we have the set of graphs with $n\ge4$ vertices, diameter $diam(G)=D=3$ and location number $n-3$,  all of them being doubly-connected and verifying $diam(\overline G)=3$. Among them, we are just interested in those graphs $G$ for which $\beta(\overline G)=n-3$. It is a routine exercise to check that as well as the path $P_4$ and the bull graph $B$, the only doubly-connected graphs of diameter 3 satisfying $\beta( G)=\beta(\overline G)=n-3$ are those belonging to $\Omega_3 \cup \Omega_3$. Hence, according to Lemma \ref{n-3D2}, to finalize the proof it suffices to check  that the only doubly-connected graph of order $4\le n \le 5$ having both itself and its complement diameter 2 is the cycle $C_5$.
\end{proof}

\begin{figure}[ht]
  \begin{center}
        \includegraphics[width=.6\textwidth]{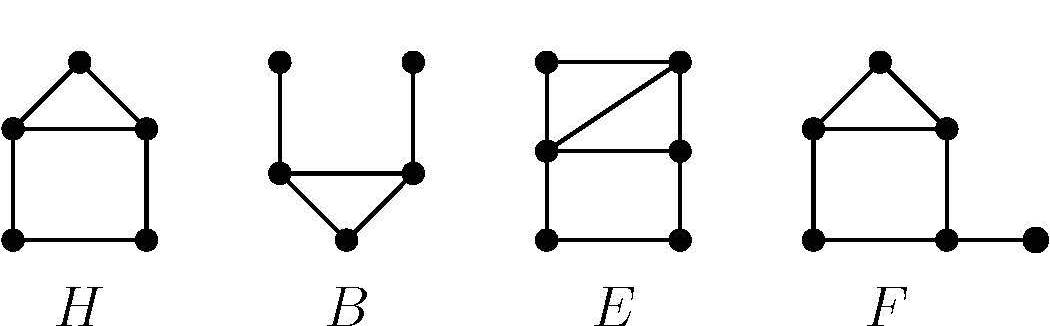}
  \end{center}
  \caption{House graph $H=\overline P_5$, bull graph $B=\overline B$, graph $E$ and graph $F=\overline E$...}
  \label{hbdf}
\end{figure}

\subsection{Metric-location-domination number}

\begin{theorem}\label{eta1} For every  nontrivial graph $G$, $3\le \eta(G)+\eta(\overline G)\le 2n-1$. Moreover,
\begin{itemize}
\item $\eta(G)+\eta(\overline G)=3$ if and only if $\{G,\, \overline G\}=\{K_2,\, \overline K_2\}$.
\item $\eta(G)+\eta(\overline G)=2n-1$ if and only if $\{G,\, \overline G\}=\{K_n,\, \overline K_n\}$.
\end{itemize}
\end{theorem}
\begin{proof}
The only nontrivial graph $G$ such that $\eta(G)=1$ is $G=K_2$, which means that for every graph $G$, $3\le \eta(G)+\eta(\overline G)$. Moreover, the equality $\eta(G)+\eta(\overline G)=3$ is only true when either $G$ or $\overline G$ is  $K_2$, since $\eta(\overline K_2)=2$. The rest of the proof is similar to that of Theorem \ref{beta1}.
\end{proof}

Given two positive integers $r,s$, let $K_2(r,s)$ denote the so-called double star, obtained after joining the central vertices of the stars $K_{1,r}$ and $K_{1,s}$. If $2\le s\le r-1$, let $K^s_{1,r}$ represent the graph obtained by adding a new vertex adjacent to $s$ leaves of the star $K_{1,r}$. Finally, $\overline K_2(r,s)$, $\overline K^s_{1,r}$ denote the complements of $K_2(r,s)$, $K^s_{1,r}$,respectively, and graphs $B$, $H$, $E$ and $F$ are shown in Figure \ref{hbdf}.

\begin{theorem}\label{eta2} For any doubly-connected graph $G$ with $n\geq5$,  $4\le \eta(G)+\eta(\overline G)\le 2n-5$. Moreover,
\begin{itemize}
\item $\eta(G)+\eta(\overline G)=4$ if and only if $G\in\{P_5,C_5,B,H,E,F\}$.
\item $\eta(G)+\eta(\overline G)=2n-5$ if and only if $G\in\{K_2(r,s),\overline K_2(r,s),K^s_{1,r},\overline K^s_{1,r}\}$.
\end{itemize}
\end{theorem}
\begin{proof}
Every  doubly-connected graph $G$ of order at least 5 satisfies $2\le\eta(G)$, since the unique nontrivial graph such that $\eta(G)=1$ is $G=P_2$. In other words, for every nontrivial doubly-connected graph $G$, $4\le \eta(G)+\eta(\overline G)$. In \cite{cahemopepu12}, it was proved that there are exactly 51 connected graphs satisfying $\eta(G)=2$, any of them having an order between 3 and 8. It is a routine exercise to check that the only doubly-connected graphs $G$ with order at least 5 of this family whose complement verify also  $\eta(\overline G)=2$ are exactly the graphs belonging to the set $\{P_5,C_5,B,H,E,F\}$.

In \cite{heoe}, it was proved that if $G$ is a connected graph such that  $\eta(G)=n-1$, then $G$ is either the complete graph $K_n$ or the star $K_{1,n-1}$. Hence, every doubly-connected graph $G$ of order $n\geq4$ satisfies $\eta(G)\le n-2$, since  both $\overline K_n$ and $\overline K_{1,n-1}$ are not connected. Also in \cite{heoe}, all connected graphs $G$ for which $\eta(G)=n-2$ were completely characterized. It is a routine exercise to check that the complement of any graph $G$ verifying $\eta(G)=n-2$ is not connected unless $G$ is  either a double star $K_2(r,s)$ or a graph $K^s_{1,r}$. As $\eta(\overline K_2(r,s))=\eta(\overline K^s_{1,r})=n-3$, we conclude first, that every doubly-connected graph $G$ of order $n\geq5$ satisfies $\eta(G)+\eta(\overline G)\le 2n-5$ and second, that these four families are the only ones attaining this upper bound.

\end{proof}

\subsection{Location-domination number}


\begin{theorem}\label{lambda2} For every  nontrivial graph $G$, $3\le \lambda(G)+\lambda(\overline G)\le 2n-1$. Moreover,
\begin{itemize}
\item $\lambda(G)+\lambda(\overline G)=3$ if and only if $\{G,\, \overline G\}=\{K_2,\, \overline K_2\}$.
\item $\lambda(G)+\lambda(\overline G)=2n-1$ if and only if $\{G,\, \overline G\}=\{K_n,\, \overline K_n\}$.
\end{itemize}
\end{theorem}
\begin{proof}
It  is similar to that of Theorem \ref{eta1}.
\end{proof}

\begin{theorem}\label{lambda3} For any doubly-connected graph $G$ with $n\geq5$,  $4\le \lambda(G)+\lambda(\overline G)\le 2n-5$. Moreover,
\begin{itemize}
\item $\lambda(G)+\lambda(\overline G)=4$ if and only if $G\in\{P_5,C_5,B,H\}$.
\item $\lambda(G)+\lambda(\overline G)=2n-5$ if and only if $G\in\{K_2(r,s),\overline K_2(r,s),K^s_{1,r},\overline K^s_{1,r}\}$.
\end{itemize}
\end{theorem}
\begin{proof}
Every  doubly-connected graph $G$ of order at least 5 satisfies $2\le\lambda(G)$, since the unique nontrivial graph such that $\lambda(G)=1$ is $G=P_2$. In other words, for every nontrivial doubly-connected graph $G$, $4\le \lambda(G)+\lambda(\overline G)$. In \cite{cahemopepu12}, it was proved that there are exactly 16 connected graphs satisfying $\lambda(G)=2$, any of them having an order between 3 and 5. It is a routine exercise to check that the only doubly-connected graphs $G$ of this family whose complement verify also $\lambda(\overline G)=2$ are the 5-path $P_5$, the 5-cycle $C_5$, the bull graph $B$ and the house graph  $H$ (see Figure \ref{hbdf}).
The rest of the proof is similar to that of Theorem \ref{eta2} since for every graph $G$, if $\lambda(G)=n-1$, then $G$ is either the complete graph $K_n$ or the star $K_{1,n-1}$ \cite{slater88} and, $\lambda(G)=n-2$ if and only if $\eta(G)=n-2$ \cite{cahemopepu12}.
\end{proof}

Observe that the only doubly-connected graph of order at most 4 is $P_4$, and notice also that $\overline P_4=P_4$ and $\eta(P_4)=\lambda(P_4)$, which means that $\eta(P_4)+\eta(\overline P_4)=\lambda(P_4)+\lambda(\overline P_4)=4$.

Finally, we present a further Nordhaus-Gaddum-type result for the parameter $\lambda$, which is a direct consequence of the fact that LD-sets in a graph $G$ are very strongly related to LD-sets in its complement $\overline G$.

\begin{proposition}\label{lapropo}
If $S$ is an LD-set of a graph $G$ then $S$ is also an LD-set of $\overline G$, unless there exists a vertex $w\in V\setminus S$  such that $ S\subseteq  N_G(w)$, in which case $S\cup\{w\}$ is an LD-set of $\overline G$.
\end{proposition}
\begin{proof}
Take $u,v\in V\setminus S$. Since $S$ is an LD-set of $G$, $\emptyset\not=S\cap N_G(u)\not= S\cap N_G(v)\not=\emptyset$. Hence,
$S\cap N_{\overline G}(u)=S\setminus S\cap N_G(u)\not= S \setminus S\cap N_G(v)= S\cap N_{\overline G}(v)$. At this point we distinguish two cases: if there exists a vertex $w\in V\setminus S$  such that $S\subseteq  N_G(w)$, or equivalently, such that $S\cap N_{\overline G}(w)=\emptyset$, then it is unique as $c_{\stackrel{}{S}}(w)=(1\ldots1)$, and thus $S\cup\{w\}$ is an LD-set. Otherwise, for every vertex $w$, $S\cap N_{\overline G}(w)\not=\emptyset$, which means that $S$ is also an LD-set of $\overline G$.
\end{proof}

\begin{theorem}\label{parecidos}
For every graph $G$,  $|\lambda(G)-\lambda(\overline G)|\le 1.$
\end{theorem}

\begin{proof}
According to Proposition \ref{lapropo}, if $S$ is a $\lambda$-code of $G$, then there exists an LD-set of $\overline G$ of cardinality at most $\lambda(G)+1$, which means that $\lambda(\overline G)\le \lambda(G)+1$. Similarly, it is derived that $\lambda(G)\le \lambda(\overline G)+1$, as $G=\overline{\overline G}$.
\end{proof}

\begin{corollary} Every graph $G$ satisfies: $2\lambda(G)-1\leq\lambda(G)+\lambda(\overline G)\leq 2\lambda(G)+1.$
\end{corollary}

\section*{Acknowledgements}
Research partially supported by grants Gen.Cat.DGR 2009SGR1040, MEC MTM2009-07242, MTM2011-28800-C02-01 and by the ESF EUROCORES programme EuroGIGA -ComPoSe IP04- MICINN Project EUI-EURC-2011-4306.



\begin{thebibliography}{99}

\bibitem{aoha12} M. Aouchiche, P. Hansen,
A survey of Nordhaus-Gaddum type relations,
\emph{Discrete Appl. Math.},
doi:10.1016/j.dam.2011.12.018.



\bibitem{cahemopepu12}
J. C\'{a}ceres, C. Hernando, M. Mora, I. M. Pelayo, M. L. Puertas,
Locating dominating codes: bounds and extremal cardinalities, 2012.
http://arxiv.org/abs/1205.2177.




\bibitem{cherjooe00} G. Chartrand, L. Eroh, M. A. Johnson, O. R. Oellermann,
Resolvability in graphs and the metric dimension of a graph,
Discrete Appl. Math. 105 (1-3) (2000) 99--113.

\bibitem{chlezh11} Chartrand, G., Lesniak, L., P. Zhang,
Graphs and Digraphs, fifth edition,
{CRC Press}, Boca Raton (FL), (2011). 


\bibitem{cohe77} E. J. Cockayne, S. T. Hedetniemi,
Towards a theory of domination in graphs,
Networks 7 (3) (1977) 247--261.


\bibitem{colebu81} D. G. Corneil, H. Lerchs, L. Stewart Burlingham,
Complement reducible graphs,
Discrete Appl. Math. 3 (3) (1981) 163--174.

\bibitem{hame} F. Harary and R. A. Melter,
On the metric dimension of a graph,
Ars Combin. 2 (1976) 191--195.


\bibitem{hahesl98} T. W. Haynes, S. T. Hedetniemi, P. J. Slater,
Fundamentals of domination in graphs,
Marcel Dekker, New York, 1998.


\bibitem{hejoso11} M. A. Henning, E. J. Joubert, J. Southey,
Nordhaus-Gaddum bounds for total domination,
Appl. Math. Lett. 24 (2011) 987--990.


\bibitem{heoe} M. A. Henning, O. Oellermann,
Metric-locating-dominating sets in graphs,
Ars Combin. 7 (2004) 129--141.


\bibitem{ours2} C. Hernando, M. Mora, I. M. Pelayo, C. Seara, D. R. Wood,
Extremal graph theory for metric dimension and diameter,
Electron. J. Combin. 17 (2010) R30.


\bibitem{jowo} G. Joret, D. R. Wood,
Nordhaus-Gaddum for treewidth,
European J. of Combin. 33 (2012) 488-490.


\bibitem{lobstein}
A. Lobstein, Watching systems, identifying, locating-dominating ans discriminating codes in graphs, 
\newline http://www.infres.enst.fr/~lobstein/debutBIBidetlocdom.pdf



\bibitem{slater75} P. J. Slater,
Leaves of trees,
Congressus Numerantium 14 (1975) 549--559.


\bibitem{slater88} P. J. Slater,
Dominating and reference sets in a graph,
J. Math. Phys. Sci. 22 (1988) 445--455.


\end{thebibliography}
\end{document}